\documentclass[a4paper,11pt,fleqn]{article}

\usepackage{amsfonts,amsmath,amsthm,amscd,amssymb,latexsym,cite,verbatim,texdraw,floatflt,caption2,pb-diagram}

\usepackage[mathscr]{eucal}
\usepackage{xcolor}

\newcommand{\ts}{\textstyle}
\newcommand{\tsfrac}[2]{{\ts\frac{#1}{#2}}}

\newtheorem{theorem}{Theorem}[section]

\newtheorem{remark}{Remark}[section]
\newtheorem{proposition}{Proposition}[section]
\newtheorem{corollary}{Corollary}[section]
\theoremstyle{definition}

\newcommand{\keywords}{\textbf{Keywords: } }

\newcommand{\subjclass}{\textbf{Mathematics Subject Classification (2010):} }
\renewcommand{\abstract}{\textbf{Abstract.} }

\numberwithin{equation}{section}

\begin{document}

%%%%%%%%%%%%%%%%%%%%%%%%%%%%%%%%%%%%%%%%%%%%%%%%%%%%%%%%%%%%%%%%%%

%\title{Approximation properties of the generalized de la Vall\'{e}e Poussin operator  on the unit disc and their applications}
\title{\bf Best weighted approximation of some kernels\\ on the real axis}

\author{Stanislav~Chaichenko${}^{1}$, Viktor~Savchuk${}^{2}$, Andrii~Shydlich${}^{2,3}$}

\maketitle

       {${}^{1}$~Donbas State Pedagogical University, Sloviansk, Ukraine\\
        ${}^{2}$~Institute of Mathematics of the National Academy of Sciences of Ukraine,
                  Kyiv,  Ukraine\\
        ${}^{3}$~National University of Life and Environmental Sciences of Ukraine,
                  Kyiv, Ukraine}\\

{\it E-mails: s.chaichenko@gmail.com, vicsavchuk@gmail.com, shidlich@gmail.com}

\bigskip

\begin{abstract}

We calculate the exact value and find the polynomial of the best weighted polynomial approximation
 of kernels of the form $\tsfrac {A+Bt}{(t^2+\lambda^2)^{s+1}}$, where $A$ and $B$ are fixed complex numbers, $\lambda>0$,   $s\in {\mathbb N}$,  in the mean square metric.
\end{abstract}

\bigskip



\section{Introduction}

 This paper considers the problem of weighted polynomial approximation in the mean square metric of kernels of the form
 \begin{equation}\label{eq:Intro1}
  {\mathcal K}_{\lambda,s}(t):={\mathcal K}_{\lambda,s}(t,A,B)=\tsfrac{A+Bt}{(t^2+\lambda^2)^{s+1}}
 \end{equation}
where $A$ and $B$ are fixed complex numbers, $\lambda>0$,
and $s$ is a positive integer ($s\in \mathbb{N}$). A lot of works are devoted to the solutions of extremal problems of the best approximation of kernels of various types in uniform and integral metrics
%(see, for example, the monographs \cite{Ahiezer_1956, Dzyadyk_Shevchuk_2008, Stepanets_2005} and  paper %\cite{Savchuk_Savchuk_2024}, which provides
(see, for example,  \cite{Ahiezer_1956, Dzyadyk_Shevchuk_2008, Savchuk_Savchuk_2024, Stepanets_2005}, which provide
a bibliography on this topic). In particular,  Ahiezer \cite{Ahiezer_1956}  was the first to establish the exact values of the best weighted approximations on the real axis of kernels \eqref{eq:Intro1} in the case of $s=0$ and found the extreme polynomials. He proved that for
any fixed algebraic polynomial
$$
    P_n(t)=P_n(0) \prod_{k=1}^n \left( 1-\tsfrac{t}{b_k}\right), \quad
    (\mathop {\rm Im } b_k\not=0)
$$
of degree $n$ and for arbitrary real numbers
$A,B$ and $\lambda>0$, the following equality holds:
 $$
    \min_{c_k} \max_{-\infty<t<\infty} \tsfrac{1}{|P_n(t)|}
    \Big| \tsfrac{A+Bt}{t^2+\lambda^2}-\sum_{k=0}^{n-1} c_k t^k\Big|
    =\tsfrac{\sqrt{A^2+\lambda^2B^2}}{2\lambda^2} \exp \Big\{ -\tsfrac{\lambda}{\pi}
    \int_{-\infty}^\infty \tsfrac{\ln |P_n(x)|}{x^2+\lambda^2}{\rm d}x\Big\}.
$$

The solution to the corresponding extremal problem of approximation by rational functions with given poles in a uniform metric was obtained by Bernstein (see, e.g. \cite{Bernstein_1964}).

The Cauchy kernel is the function
%$\mathcal C: \mathbb C^2\rightarrow\mathbb C$
$$\label{Couchi-kern}
(z,t) \mapsto \frac{1}{z-t}, \quad (z,t)\in\mathbb C^2\setminus\Delta,
$$
where $\Delta:=\{(z,t)\in\mathbb C^2 : z=t\}$.

Dzhrbashyan \cite{Dzhrbashyan_1974} developed a method that provides solutions to extremal problems about the best rational approximations of the Cauchy kernel in the case where $z$ is fixed with $\mathrm{Im}z\not=0$ and $-\infty <t< \infty$, both in the uniform and mean-square metrics on the real axis $\mathbb R$. The method is based on the use of certain orthogonal and bi-orthogonal systems of rational functions on $\mathbb{R}$ with fixed poles in the upper half-plane $\mathbb C_+:=\{z\in\mathbb C : \mathrm{Im}z>0\}$.

Using the method of Dzhrabashyan, Voskanian \cite{Voskanyan_1979} improved upon the result of  Akhiezer,   solved similar extremal problems  for kernels \eqref{eq:Intro1} in the case of $s=0$ in the mean square metric, and also considered the corresponding problems of approximating such kernels by rational functions with fixed poles.

In   \cite{Savchuk_Chaichenko_Shydlich_2024}, similar extremal problems of the best weighted polynomial approximation of kernels  \eqref{eq:Intro1} in the case of $s=1$ in the mean square metric were solved.

This paper provides solutions to  extremal problems of best weighted polynomial approximation of kernels \eqref{eq:Intro1} for arbitrary positive integer $s$  in the mean square metric. In particular, the extremal polynomial is found explicitly and the exact values of the best weighted approximation of such kernels in the mean square metric are calculated.

For $A=s=0$ and $B=\tsfrac{1}{\pi}$, the kernel ${\mathcal K}_{\lambda,s}(t,A,B)$ coincides with the Poisson kernel $P_t(\lambda)=\tsfrac{1}{\pi}\tsfrac{t}{\lambda^2+t^2}$ as a function of the variable $\lambda$.

A special case of kernels \eqref{eq:Intro1} are integral kernels defined in the upper half-plane by the well-known biharmonic Poisson integrals
  \[
    \mathcal{B}(f;t;\lambda):= \tsfrac{2 \lambda^3}{\pi} \int_{-\infty}^\infty
    \tsfrac{f(x+t)}{(x^2+\lambda^2)^2}~\mathrm{d}x,
  \]
which give the solution of the biharmonic equation
  $$
    \nabla^2 (\nabla^2 U)=0, \quad
    \nabla := \tsfrac{\partial^2}{\partial x^2}+\tsfrac{\partial^2}{\partial \lambda^2},
  $$
  in the upper half-plane of the complex plane ($\lambda>0$) under the boundary conditions
  $$
    \lim_{\lambda \to 0+} U(x, \lambda)=f(x), \quad
    \lim_{\lambda \to 0+} \tsfrac{\partial}{\partial \lambda} U(x, \lambda)=0.
$$
The approximation properties of similar Poisson integrals   were investigated in \cite{Kharkevych_Shutovskyi_2024, Shutovskyi_Zhyhallo_2025}.
We also note that in \cite{Savchuk_Savchuk_2024}, weighted combinations of the Cauchy-Szeg\"{o} kernel and its derivatives on the unit circle were studied, and, in turn, kernels \eqref{eq:Intro1} can be considered as a special case of linear combinations of Poisson kernels and their derivatives.
Incidentally, the properties of Poisson kernels are often used in solving problems in physics, mechanics, electrical engineering, etc.

\section{{Preliminaries}}

Let  ${\bf a}:=\{a_k\}_{k=1}^{\infty}$  be a sequence of complex numbers on
the upper half-plane
   $\mathbb C_+:=\{z=x+{\rm i}y\in\mathbb C : \mathop{\rm Im}z>0\}$. Consider the system of functions
 \[
     \Phi_j(z):=\tsfrac{\sqrt{\mathop{\rm Im} a_j}}{z-\overline a_{j}}B_j(z),\quad j=1,2,\ldots,
 \]
where
$$
    B_1(z):=1,\quad B_j(z):=\prod_{k=1}^{j-1}\chi_k
    \tsfrac{z-a_k}{z-\overline a_k},\quad j=2,3,\ldots,
$$
are the Blaschke $j$-product with zeros at the points $a_k,$ $\chi_k:=\tsfrac{|1+a^2_k|}{1+a^2_k}$, $k=1,2,...$.

The system of functions $\{\Phi_j\}_{j=1}^\infty$ was introduced by Dzhrbashyan \cite{Dzhrbashyan_1974}
by analogy, as was done by  Takenaka and  Malmquist (see \cite{Takenaka_1925, Malmquist_1925})
 in the case of the Hardy space $H_2$ in the unit disc. In particular, it was shown in \cite{Dzhrbashyan_1974} that this system is orthonormal on the real axis $\mathbb R$, i.e.
$$
    \tsfrac{1}{\pi} \int_{-\infty}^\infty
    \Phi_j (x)\overline{\Phi_k(x)}~\mathrm{d}x=
    \left\{\begin{matrix}0, \hfill& j\not=k,\\
                    1,\hfill& j=k,
                    \end{matrix}\right.\quad j,k=1,2,\ldots,
$$
as well as the following statement is true:

\begin{proposition}[{\cite[Theorem 1]{Dzhrbashyan_1974}}]
\label{Prop:Preliminaries1}
 For arbitrary $z, \zeta\in\mathbb C,$ $z\not=\overline\zeta$ and
any  positive integer $n$
\begin{equation}\label{eq:Preliminaries1} %2
    \tsfrac{1}{2 {\mathrm i}(\overline\zeta-z)} =\sum_{j=1}^{n}\overline{\Phi_j(\zeta)}\Phi_j(z)+
    \tsfrac{\overline{B_{n+1}(\zeta)}B_{n+1}(z)}{2 {\mathrm i}(\overline\zeta-z)}.
\end{equation}
\end{proposition}

The equality (\ref{eq:Preliminaries1}) is an analogue of the Christoffel-Darboux formula for orthogonal polynomials and is known in the literature as the Dzhrbashyan identity. A simple proof of Proposition \ref{Prop:Preliminaries1} can also be found in \cite{Savchuk_Chaichenko_2015}.

A function $f$, analytic in the half-plane $\mathbb{C}_+$, belongs to the Hardy class $H_2(\mathbb{C}_+),$ if
\[
    \|f\|_{H_2(\mathbb{C}_+)}:= \sup_{y>0}\Big(\tsfrac{1}{\pi} \int_{-\infty}^\infty|f(x+{\mathrm i}y)|^2~\mathrm{d}x\Big)^{1/2}<\infty.
\]

It is known \cite[Theorem 2]{Kryloff_1939} (see also \cite[p. 291]{Mashreghi_2009}) that every function $f$ from the space $H_2(\mathbb{C}_+ )$ has non-tangent limiting values
$$
    f(x)=\lim_{y \to 0+} f(x+{\mathrm i}y), \quad x+{\rm i}y\in {\mathbb C}_+,
$$
almost everywhere on the real axis $\mathbb R=\partial\mathbb C_+$, which belong to the space $L_p(\mathbb R)$. Furthermore,
\[
    \|f\|_{H_2(\mathbb{C}_+)}=
    \|f\|_2:=\Big(\tsfrac{1}{\pi}\int_{-\infty}^\infty|f(t)|^2~\mathrm{d}t\Big)^{1/2},
\]
and the Cauchy's integral formula holds:
\begin{equation}\label{eq:Preliminaries2}
    f(z)=\tsfrac{1}{2\pi {\mathrm i}} \int_{-\infty}^\infty \tsfrac{f(t)}{t-z}~\mathrm{d}t,
    \quad \forall z\in\mathbb C_+.
\end{equation}

The following statement is a consequence of the combination of   Proposition \ref{Prop:Preliminaries1} and  formula \eqref{eq:Preliminaries2}, and was first obtained by Dzhrbashyan in \cite{Dzhrbashyan_1974}.

\begin{proposition}[{\cite[Corollary 1.3]{Dzhrbashyan_1974}}]\label{Prop:Preliminaries2}
If  $f \in H_2(\mathbb{C}_+),$ then for any $n \in \mathbb{N}$  and $z\in\mathbb C_+$
 \[
    f(z)=\sum_{j=1}^n \widehat{f}(j)\Phi_j(z)+ \tsfrac{B_{n+1}(z)}{2\pi {\mathrm i}}\int_{-\infty}^\infty \tsfrac{f(t)\overline{B_{n+1}(t)}}{t-z}~\mathrm{d}t,
 \]
where
$$
    \widehat{f}(j):=\tsfrac{1}{\pi}\int_{-\infty}^\infty
    f(t) \overline{\Phi_j(t)}~\mathrm{d}t, \quad j=1,2,\ldots~.
$$
\end{proposition}

Let $A$ and $B$ be any complex numbers, $\lambda>0$ and $n,s\in {\mathbb N}$. Set
 \[
    R_{n,s}^{\lambda}(z):=R_{n,s}^{\lambda}(z, A,B,{\bf a})=\tsfrac {A+Bz}{(z^2+\lambda^2)^{s+1}
    \tau_n(z,{\bf \bar{a}})},\quad z\in \mathbb{C}_+,
 \]
where  $\tau_n(z,{\bf a}):=\prod\limits_{k=1}^n (z-a_k)$ and  ${\bf \bar{a}}:=\{\bar{a}_k\}_{k=1}^{\infty}$.

Let $S_n(R_{n,s}^{\lambda})$ denote the $n$-th partial Fourier sum of the function $R_{n,s}^{\lambda}$ with respect to the system $\{\Phi_j\}_{j=1}^\infty$, i.e.,
  \[
  S_n(R_{n,s}^{\lambda})(z)=\sum_{j=1}^n \widehat{R}_{n,s}^{\lambda}(j)\Phi_j(z),\quad  z\in \mathbb{C}_+,
  \]
where $\widehat{R}_{n,s}^{\lambda}(j)=\tsfrac 1\pi\int_{-\infty}^\infty R_{n,s}^{\lambda}(t,\lambda)\overline{\Phi_j(t)}{\mathrm d}t$, $j=1,2,\ldots$

For a fixed  $k=0,1,\ldots$, we denote
  \begin{equation} \label{eq:Preliminaries3}
    \nu_k(z,{\bf a}):=\sum_{|{\bf k}|=k}
  \prod_{j=1}^n  \tsfrac {1}{(z- a _j)^{k_j}}.
\end{equation}
where the sum is taken over all possible vectors ${\bf k}=(k_1, \ldots, k_n)$   with non-negative integer coordinates  for which  $|{\bf k}|:=k_1+ \ldots+k_n=k$.

\section{{Main results}}

The following theorem gives a representation for the deviation of the kernel $R_{n,s}^{\lambda}(z)$ from the partial sum $S_n(R_{n,s}^{\lambda})(z)$ at each point $z\in \mathbb{C}_+$.

\begin{theorem} \label{Theo:Main_Results1}
  Let  $A$ and $B$ be fixed complex numbers, and $\lambda>0.$ Then
  for any  $z\in \mathbb{C}_+$ and   $n,s\in \mathbb{N}$  the following identity holds:
  {
  %%%%%%%%%%%%%%%%%%%%%%%%%%%%%%%%%%%%%%%%%%%%%%%%%%%%%%%%%%%%%%%%%%%%%%%%%%%%%%%%%%%%%%%%%%%%%%%%%
  \begin{equation}\label{eq:Main_Results1}
    R_{n,s}^{\lambda}(z)-S_n(R_{n,s}^{\lambda})(z)=
    \sum_{l=0}^{s} \Big(\tsfrac {D_l}
    {({\mathrm i}\lambda-z)^{l+1}} +
   \tsfrac{ {\tau}_n(z,{\bf a})}{\tau_n(z,{\bf \bar{a}})} \tsfrac {\overline{D}_l}
   {(-{\mathrm i}\lambda-z)^{l+1}}\Big),
\end{equation}
where
  \begin{equation}\label{eq:Main_Results1_0}
   D_l:=\displaystyle{\sum_{j=0}^{s-l}}
   \tsfrac {(2{\mathrm i}\lambda)^{s-j}{s+j \choose s}\nu_{s-l-j}({\mathrm i}\lambda,{\bf \bar{a}})}
   {(2\lambda)^{2s+1}\tau_n({\mathrm i}\lambda,{\bf \bar{a}})}
   \left( \tsfrac {B\lambda (s(s+j)-2 j)}{s(s+j)}- {\mathrm i}A \right).
  \end{equation}
  }
\end{theorem}

%%%%%%%%%%%%%%%%%%%%%%%%%%%%%%%%%%%%%%%%%%%%%%%%%%%%%%%%%%%%%%%%%%%%%%%%%%%%%%%%%%%%%%%%%%%%

%%%%%%%%%%%%%%%%%%%%%%%%%%%%%%%%%%%%%%%%%%%%%%%%%%%%%%%%%%%%%%%%%%%%%%%%%%%%%%%%%%%%%%%%%%

\begin{proof}[Proof of Theorem~\ref{Theo:Main_Results1}]

Let $z\in \mathbb{C}_+$. Consider the following Cauchy-type integral
\begin{equation} \label{eq:Main_Results2}
 {\cal J}_n^\lambda(z):=\textstyle{\tsfrac 1{2\pi {\rm i}}}
 \displaystyle{\int_{-\infty}^\infty}
  \tsfrac{R_{n,s}^{\lambda}(t)}{t-z}{\mathrm d}t=\tsfrac 1{2\pi {\rm i}}
 \displaystyle{\int_{-\infty}^\infty}
  \tsfrac{A+Bt}{(t^2+\lambda^2)^{s+1} (t-z)\tau_n(t,{\bf \bar{a}})} {\mathrm d}t .
\end{equation}
The integrand as a function of the complex variable $t\in \mathbb{C}_+$ has  poles of  order 1 at the point $t_1=z$  and  of order $s+1$ at the point $t_2={\mathrm i}\lambda $, and for $|t|\to +\infty$,  it is   ${\mathcal O}(|t|^{-n-2s-2})$. Therefore, based on  Cauchy residue theorem, we have
\begin{equation} \label{eq:Main_Results3}
 {\cal J}_n^\lambda(z)=\Big(\mathop{\rm res}\limits_{t=t_1}+\mathop{\rm res}\limits_{t=t_2}\Big)
  \tsfrac{A+Bt}{(t^2+\lambda^2) ^{s+1}(t-z)\tau_n(t,{\bf \bar{a}}) },
\end{equation}
where
\begin{equation} \label{eq:Main_Results4}
  \mathop{\rm res}\limits_{t=t_1}\!
  \tsfrac{A+Bt}{(t^2+\lambda^2) ^{s+1}(t-z)\tau_n(t,{\bf \bar{a}})}
  =
 \tsfrac{A+Bz}{(z^2+\lambda^2)^{s+1}\tau_n(z,{\bf \bar{a}})}=R_{n,s}^{\lambda}(z),
\end{equation}
and
\begin{equation} \label{eq:Main_Results5}
     \mathop{\rm res}\limits_{t=t_2}
  \textstyle{\tsfrac{A+Bt}{(t^2+\lambda^2)^{s+1}(t-z)\tau_n(t,{\bf \bar{a}})}=\tsfrac 1 {s!}}
  \lim\limits_{t\to {\mathrm i}\lambda}
  \textstyle{\Big[\tsfrac{A+Bt}{(t+{\mathrm i}\lambda)^{s+1}(t-z)\tau_n(t,{\bf \bar{a}})} \Big]^{(s)}_t }.
\end{equation}
Since for any $k=0, 1,2,\ldots$,
 \[
   (A+Bt)^{(k)} =\left\{\begin{matrix}
  A+Bt, \hfill& k=0,\\
  B,\hfill& k=1,\\
  0,\hfill& k>1,
  \end{matrix}\right.
  \]
  and
 \[
 \textstyle{\left[\tsfrac 1{ t-z }\right]^{(k)}_t=\tsfrac {(-1)^k k!}{(t-z)^{k+1}},}
 \quad \textstyle{\left[\tsfrac 1{(t+{\mathrm i}\lambda)^{s+1}}\right]^{(k)}_t  =\tsfrac {(-1)^k (s+k)!}{s!(t+{\mathrm i}\lambda)^{s+k+1}} },
 \]
and  by Leibniz rule,
  \begin{eqnarray}\nonumber
    \Big[\tsfrac 1{\tau_n(t,{\bf \bar{a}})}\Big]^{(k)}_t =
     \sum_{|{\bf k}|=k}\tsfrac {k!}{k_1! k_2!\cdots k_n!}
     \prod_{j=1}^n \Big[\tsfrac {1}{t-\overline{a}_j}\Big]^{(k_j)}
      =
      \sum_{|{\bf k}|=k}
      \tsfrac {k!}{k_1!k_2!\cdots  k_n!}
      \prod_{j=1}^n
      \tsfrac {(-1)^{k_j} k_j!}{(t-\overline{a}_j)^{k_j+1}}
      =
      \tsfrac {(-1)^{k} k!\, \nu_k(t,{\bf \bar{a}})}{\tau_n(t,{\bf \bar{a}})},
  \end{eqnarray}
we get
 \begin{eqnarray}\nonumber
   \left[\tsfrac{A+Bt}{(t+{\mathrm i}\lambda)^{s+1}(t-z)\tau_n(t,{\bf \bar{a}})} \right]^{(s)}_t  
    &=&
     \sum_{|{\bf s}|=s}
      \tsfrac {(A+Bt)s!}{  s_1!s_2!s_3!}
      \tsfrac {(-1)^{s_1} {s_1}!}{(t-z)^{s_1+1}}
       \tsfrac {(-1)^{s_3} (s+s_3)!}{s!(t+{\mathrm i}\lambda)^{s+s_3+1}}
       \tsfrac {(-1)^{s_2} {s_2}! \,\nu_{s_2}(t,{\bf \bar{a}})}{\tau_n(t,{\bf \bar{a}})}
\\ \nonumber
      &&+ \sum_{|{\bf s}|=s-1}
  \textstyle{\tsfrac {B(s-1)!}{  s_1!s_2!s_3!}
     \tsfrac {(-1)^{s_1} s_1!}{(t-z)^{s_1+1}}
      \tsfrac {(-1)^{s_3} (s+s_3)!}{s!(t+{\mathrm i}\lambda)^{s+s_3+1}}
   \tsfrac {(-1)^{s_2} s_2!\, \nu_{s_2}(t,{\bf \bar{a}})}{\tau_n(t,{\bf \bar{a}})}}
   \\ \nonumber
    &=&
    \tsfrac {(-1)^{s} s!}{\tau_n(t,{\bf \bar{a}}) }\Bigg[
     \sum_{|{\bf s}|=s}
      \tsfrac { (A+Bt) {s+s_3\choose s} \nu_{s_2}(t,{\bf \bar{a}})}{(t-z)^{s_1+1}(t+{\mathrm i}\lambda)^{s+s_3+1}}  -
      \sum_{|{\bf s}|=s-1}\tsfrac { B{s+s_3\choose s} \nu_{s_2}(t,{\bf \bar{a}}) }{s(t-z)^{s_1+1}(t+{\mathrm i}\lambda)^{s+s_3+1}}
      \Bigg],
 \end{eqnarray}
 where ${\bf s}=(s_1,s_2,s_3)$.  Therefore,
  \begin{eqnarray}\nonumber
    \left[\tsfrac{A+Bt}{(t+{\mathrm i}\lambda)^{s+1}(t-z)\tau_n(t,{\bf \bar{a}})} \right]^{(s)}_t
   &=&
    \tsfrac {(-1)^{s}s!}{\tau_n(t,{\bf \bar{a}})}
    \Bigg[\sum_{s_1=0}^{s}\sum_{s_2=0}^{s-s_1}
      \tsfrac {(A+Bt){2s-s_1-s_2 \choose s} \nu_{s_2}(t,{\bf \bar{a}}) }
      {(t-z)^{s_1+1}(t+{\mathrm i}\lambda)^{2s-s_1-s_2+1}}
      \\ \nonumber
   &&
    -
     \sum_{s_1=0}^{s-1}\sum_{s_2=0}^{s-1-s_1}
     \tsfrac {B {2s-s_1-s_2 \choose s}(s-s_1-s_2)\nu_{s_2}(t,{\bf \bar{a}}) }
     {s(2s-s_1-s_2)  (t-z)^{s_1+1}(t+{\mathrm i}\lambda)^{2s-s_1-s_2}}\Bigg].
  \end{eqnarray}
Since the expression in the second sum vanishes  for $s_2=s-s_1$, as well as  for $s_1=s$ and $s_2=0$, we can rewrite the last formula as follows
 \begin{eqnarray}\nonumber
   \left[\tsfrac{A+Bt}{(t+{\mathrm i}\lambda)^{s+1}(t-z)\tau_n(t,{\bf \bar{a}})} \right]^{(s)}_t
    &=&
    \tsfrac {(-1)^{s}s!}{\tau_n(t,{\bf \bar{a}})}
    \Bigg[\sum_{s_1=0}^{s}\sum_{s_2=0}^{s-s_1}
    \tsfrac {(A+Bt){2s-s_1-s_2 \choose s}\nu_{s_2}(t,{\bf \bar{a}}) }
    {(t-z)^{s_1+1}(t+{\mathrm i}\lambda)^{2s-s_1-s_2+1}}
   \\ \nonumber
    && -
    \sum_{s_1=0}^{s}\sum_{s_2=0}^{s-s_1}
     \tsfrac {B {2s-s_1-s_2 \choose s}  (s-s_1-s_2)\nu_{s_2}(t,{\bf \bar{a}}) }
     {s(2s-s_1-s_2)  (t-z)^{s_1+1}(t+{\mathrm i}\lambda)^{2s-s_1-s_2}}\Bigg]
   \\
   \nonumber
   &=&
    \tsfrac {(-1)^{s}s!}{\tau_n(t,{\bf \bar{a}})}
   \sum_{s_1=0}^{s}\sum_{s_2=0}^{s-s_1}
   \tsfrac {{2s-s_1-s_2 \choose s} \nu_{s_2}(t,{\bf \bar{a}}) }
   {(t-z)^{s_1+1}(t+{\mathrm i}\lambda)^{2s-s_1-s_2+1}}
    \\ \nonumber
    &&
     \times
     \Big(A+Bt - \tsfrac{B(t+{\mathrm i}\lambda)(s-s_1-s_2)}{s(2s-s_1-s_2)}\Big).
  \end{eqnarray}
Setting $l:=s_1$ and $j:=s-s_1-s_2$, we finally  get
\begin{equation} \nonumber
    \left[\tsfrac{A+Bt}{(t+{\mathrm i}\lambda)^{s+1}(t-z)\tau_n(t,{\bf \bar{a}})} \right]^{(s)}_t
    \!\!= \tsfrac {(-1)^{s}}{\tau_n(t,{\bf \bar{a}})}
    \sum_{l=0}^{s}\sum_{j=0}^{s-l}
   \tsfrac {s!\,{s+j \choose s}\,\nu_{s-l-j}(t,{\bf \bar{a}}) }
   {(t-z)^{l+1}(t+{\mathrm i}\lambda)^{s+j+1}}
   \Big(A+Bt - \tsfrac{B(t+{\mathrm i}\lambda)j}{s(s+j)}\Big).
\end{equation}
Combining the last relation and \eqref{eq:Main_Results5}, we see that
\begin{equation} \nonumber
  \mathop{\rm res}\limits_{t=t_2}
  \textstyle{
 \tsfrac{A+Bt}{(t^2+\lambda^2)^{s+1}(t-z)\tau_n(t,{\bf \bar{a}})}} = \textstyle{\tsfrac {(-1)^{s} }{\tau_n({\mathrm i}\lambda,{\bf \bar{a}})}}
   \displaystyle{\sum_{l=0}^{s}\sum_{j=0}^{s-l} }
   \textstyle{
  \tsfrac {{s+j \choose s}\nu_{s-l-j}({\mathrm i}\lambda,{\bf \bar{a}}) }{({\mathrm i}\lambda-z)^{l+1}(2{\mathrm i}\lambda)^{s+j+1}}
\Big( A+B{\mathrm i}\lambda  - \tsfrac{2B{\mathrm i}\lambda j}{s(s+j)}\Big)}.
\end{equation}
{
Using the identity ${\mathrm i}^{2s+2}=(-1)^{s+1}$ and taking into account \eqref{eq:Main_Results1_0}, we obtain
  \begin{eqnarray}\nonumber
   \mathop{\rm res}\limits_{t=t_2}
   \tsfrac{A+Bt}{(t^2+\lambda^2)^{s+1}(t-z)\tau_n(t,{\bf \bar{a}})} \!\!
   &=&
    \!\!\sum_{l=0}^{s}\sum_{j=0}^{s-l}
     \tsfrac {{s+j \choose s}(2{\mathrm i}\lambda)^{s-j}\nu_{s-l-j}({\mathrm i}\lambda,{\bf \bar{a}}) }
     {(2 \lambda)^{2s+1}\tau_n({\mathrm i}\lambda,{\bf \bar{a}})({\mathrm i}\lambda-z)^{l+1}}
     \Big( \tsfrac{B\lambda(s(s+j)-2 j)}{s(s+j)} -{\mathrm i}A  \Big)
     \\ \label{eq:Main_Results6}
   &=&
     \sum_{l=0}^{s}\tsfrac {D_l}{({\mathrm i}\lambda-z)^{l+1}}.
  \end{eqnarray}
Substituting the relations  \eqref{eq:Main_Results4} and \eqref{eq:Main_Results6} into \eqref{eq:Main_Results3}, we see that
 \begin{eqnarray}\label{eq:Main_Results7}
   {\cal J}_n^\lambda(z)&=& R_{n,s}^{\lambda}(z)+ \sum_{l=0}^{s}\tsfrac {D_l}{({\mathrm i}\lambda-z)^{l+1}}.
  \end{eqnarray}}

On the other hand, taking $\zeta=t\in {\mathbb R}$ in the identity \eqref{eq:Preliminaries1} and using the representation
\eqref{eq:Main_Results2}, we get
 \begin{eqnarray}\nonumber
   {\cal J}_n^\lambda(z)&=&
    \sum_{k=1}^{n}
    \tsfrac {\Phi_k(z)}{\pi}
    \int_{-\infty}^\infty
    R_{n,s}^{\lambda}(t)\overline{\Phi_k(t)}
    {\mathrm d}t+
    \tsfrac{B_{n+1}(z)}{2\pi {\mathrm i}}
     \int_{-\infty}^\infty
     \tsfrac{\overline{B_{n+1}(t)}R_{n,s}^{\lambda}(t)}{t-z}{\mathrm d}t
     \\ \label{eq:Main_Results8}
   &=& S_n(R_{n,s}^{\lambda})(z)+B_{n+1}(z)U_n^\lambda(z),
\end{eqnarray}
where
 \[
   U_n^\lambda(z)=\tsfrac{1}{2\pi {\mathrm i}}\int_{-\infty}^\infty
    \tsfrac{\overline{B_{n+1}(t)}R_{n,s}^{\lambda}(t)}{t-z}{\mathrm d}t
   = \tsfrac {\prod\limits_{k=1}^n\overline{\chi}_k}{2\pi {\rm i}}
    \int_{-\infty}^\infty
    \overline{\left(\tsfrac {\tau_n(t,{\bf a})}{\tau_n(t,{\bf \bar{a}})}\right)}
    \tsfrac{A+Bt}{(t^2+\lambda^2)^{s+1}\tau_n(t,{\bf \bar{a}})}
    \tsfrac {{\mathrm d}t}{t-z}.
    \]
For $t\in {\mathbb R}$
 \[
  \overline{\Big(\tsfrac {\tau_n(t,{\bf a})}{\tau_n(t,{\bf \bar{a}})}\Big)}=
  \tsfrac {\prod\limits_{k=1}^n\overline{(t-a_k)}}
  {\prod\limits_{k=1}^n\overline{(t-\overline{a}_k)}}
  =\tsfrac {\prod\limits_{k=1}^n(t-\overline{a}_k)}
  {\prod\limits_{k=1}^n (t- a_k)}
  =\tsfrac {\tau_n(t,{\bf \bar{a}})}{\tau_n(t,{\bf a})}.
 \]
Therefore,
 \[
   U_n^\lambda(z) =
   \tsfrac {\prod\limits_{k=1}^n\overline{\chi}_k}{2\pi {\rm i}}
    \int_{-\infty}^\infty
    \tsfrac{A+Bt}{(t^2+\lambda^2)^{s+1} \tau_n(t,{\bf a})}
    \tsfrac {{\mathrm d}t}{t-z}.
 \]
The integrand, as a function of the complex variable $t$ in the lower half plane, has a pole of order $s+1$  at the point $t=-{\mathrm i}\lambda$, and for $|t| \to +\infty$ it is   ${\mathcal O}(|t|^{-n-2s-2})$. Thus, by  Cauchy residue theorem
 \[
 U_n^\lambda(z) =-\prod\limits_{k=1}^n\overline{\chi}_k\mathop{\rm res}\limits_{t=-{\mathrm i}\lambda}
  \tsfrac{A+Bt}{(t^2+\lambda^2)^{s+1}(t-z)\tau_n(t,{\bf a})}.
 \]
Consequently,
 \[
  \mathop{\rm res}\limits_{t=-{\mathrm i}\lambda}
   \tsfrac{A+Bt}{(t^2+\lambda^2)^{s+1}(t-z)\tau_n(t,{\bf a})}
  =\tsfrac 1 {s!}
  \lim\limits_{t\to -{\mathrm i}\lambda}
  \left[  \tsfrac{ A+Bt }{(t-{\mathrm i}\lambda)^{s+1}(t-z)\tau_n(t,{\bf a})} \right]^{(s)}_t .
 \]
  This similarly yields
\begin{eqnarray}\nonumber
   \mathop{\rm res}\limits_{t=-{\mathrm i}\lambda} \tsfrac{A+Bt}{(t^2+\lambda^2)^{s+1}(t-z)\tau_n(t,{\bf a})} 
   &=&
     \Bigg[\sum_{l=0}^{s}  \sum_{j=0}^{s-l}
       \tsfrac {{s+j \choose s}(-2{\mathrm i}\lambda)^{s-j} {\nu}_{s-l-j}(-{\mathrm i}\lambda,{\bf a}) }
       {(2\lambda)^{2s+1}\tau_n(-{\mathrm i}\lambda,{\bf a})(-{\mathrm i}\lambda-z)^{l+1}}
       \Big(\tsfrac {B\lambda(s(s+j)-2j)}{s(s+j)} + {\mathrm i}A \Big)\Bigg]
       \\ \nonumber
   &=&
      \sum_{l=0}^{s}   \tsfrac {\overline{D}_l}{(-{\mathrm i}\lambda-z)^{l+1}}.
\end{eqnarray}
Hence,
\begin{eqnarray}\nonumber
   U_n^\lambda(z)=
    \prod\limits_{k=1}^n\overline{\chi}_k \sum_{l=0}^{s}   \tsfrac {\overline{D}_l}{(-{\mathrm i}\lambda-z)^{l+1}}.
\end{eqnarray}
Combining this  relation and  \eqref{eq:Main_Results8}, we obtain
\begin{eqnarray}\label{eq:Main_Results9}
    {\cal J}_n^\lambda(z)=S_n(R_{n,s}^{\lambda})(z)+
   \tsfrac{\tau_n(z,{\bf a})}{\tau_n(z,{\bf \bar{a}})}
   \sum_{l=0}^{s}   \tsfrac {\overline{D}_l}{(-{\mathrm i}\lambda-z)^{l+1}} .
\end{eqnarray}
Comparing  \eqref{eq:Main_Results7} and \eqref{eq:Main_Results9}, we get \eqref{eq:Main_Results1}.
%%%%%%%%%%%%%%%%%%%%%%%%%%%%%%%%%%%%%%%%%%%%%%%%%%%%%%%%%%%%%%%%%%%%%%%%%%%%%%%%%%%%%%%%%%%%%%%%%%%%%%%%
\end{proof}

\begin{remark} \label{Rem:Main_Results1}
Note that by the definition of the system $\{\Phi_j\}_{j=1}^\infty$, the families of rational functions of the form $\Big\{\sum\limits_{j=1}^n b_j \Phi_j(z)\Big\}$, where $b_j$ are any complex numbers,  and $\Big\{\tsfrac{p_{n-1}(z)}{\tau_n (z,{\bf \bar{a}})} \Big\}$, where $p_{n-1}$ is an algebraic polynomial of degree at most  $n-1$, coincide.
 \end{remark}

\color{black}
For a fixed value of $\lambda >0$, we define the algebraic polynomial $p_{n-1}^{(\lambda,s)}(z)$ of the degree $n-1$ with respect to the variable $z$ by the equality
\begin{equation}\label{eq:Main_Results10}
    p_{n-1}^{(\lambda,s)}(z):=p_{n-1}^{(\lambda,s)}(z,{\bf \bar{a}})=\tau_n (z,{\bf \bar{a}}) S_n(R_{n,s}^{\lambda})(z), \quad n=1,2,\ldots~.
\end{equation}

\begin{corollary} \label{Cor:Main_Results1}
  Let  $A$ and $B$ be any complex numbers, $s$ be a positive integer and $\lambda>0.$ Then for any  $z\in \mathbb{C}_+$ and   $n\in \mathbb{N}$, the following identity holds:
  {
   \begin{eqnarray}\label{eq:Main_Results10_01}
    \tsfrac{A+Bz}{(z^2+\lambda^2)^{s+1}}-p_{n-1}^{(\lambda,s)}(z)=
    \sum_{l=0}^{s}\Big(\tau_n(z,{\bf \bar{a}}) \tsfrac { D_l}{({\mathrm i}\lambda-z)^{l+1}}+
    \tau_n(z,{\bf a})\tsfrac{\overline{D}_l}
     {(-{\mathrm i}\lambda-z)^{l+1}}\Big),
\end{eqnarray}
where the numbers $D_l$, $l=0,1,\ldots,s$, are given by \eqref{eq:Main_Results1_0}.
  }

 \end{corollary}

Indeed, the equality \eqref{eq:Main_Results10_01} can be easily obtained by expressing the value of $S_n(R_n)(z)$ from equality (\ref{eq:Main_Results10}) and substituting it into identity \eqref{eq:Main_Results1} of Theorem~\ref{Theo:Main_Results1}.

{
Let ${\bf a}:=\{a_k\}_{k=1}^{\infty}$ be a fixed system of points in ${\mathbb C}_+$, and let
 \[
  \rho_n(z)=\rho_n(z,{\bf a})=\rho_0 \tau_n(z,{\bf a})=\rho_0 \prod_{k=1}^{n}(z-a_k),
 \]
where $\rho_0\not=0$ is a constant.

Denote by $P_{n-1}$ the set of all algebraic polynomials with complex coefficients of degree at most  $n-1$   and consider  the quantity
 \begin{equation}\label{eq:Main_Results11}
    \mathcal{E}_n ({\mathcal K}_{\lambda,s})_{2,\rho_n}:=\inf\limits_{p\in P_{n-1}} \Big(\int_{-\infty}^{\infty}
     \left| \tsfrac{A+Bt}{(t^2+\lambda^2)^{s+1}}-
    p(t) \right|^2 \tsfrac{\mathrm{d}t}{|\rho_n(t)|^2}\Big)^{1/2}
\end{equation}
of best weighted approximation of the kernel ${\mathcal K}_{\lambda,s}$ of the form \eqref{eq:Intro1} by all possible  polynomials  from the set $P_{n-1}$ in the mean square metric
with the weight $\tsfrac{1}{|\rho_n(t)|^2}$.

\begin{theorem}\label{Theo:Main_Results2}
Let $A,B$ be  any fixed real numbers, $\lambda>0$ and $s\in {\mathbb N}$.   Then for any $n\in {\mathbb N}$
\begin{equation} \label{eq:Main_Results12}
    \mathcal{E}_n^2 ({\mathcal K}_{\lambda,s})_{2,\rho_n}=\tsfrac{4 \pi}{(2 \lambda)^{2s+3}
    |\rho_n(-\mathrm{i}\lambda)|^2}\sum_{k=0}^{s}\sum_{l=0}^{s}   \textstyle{{l+k \choose  l   }} G_k\overline{G}_l,
\end{equation}
where
\begin{equation} \label{eq:Main_Results13}
  G_k=G_k(\lambda,{\bf a})= \sum_{j=0}^{s-k}
  \textstyle{
  \tsfrac {{s+j \choose s}{\mathrm i}^{s-k-j}\nu_{s-k-j}({\mathrm i}\lambda,{\bf \bar{a}}) }{(2\lambda)^{k+j}}
  \Big( \tsfrac {B\lambda(s(s+j)-2j)}{s(s+j)}- {\mathrm i}A \Big).}
\end{equation}
 The infimum on the right-hand side of \eqref{eq:Main_Results11} is attained for the polynomial $p_{n-1}^{(\lambda,s)} $ of the form \eqref{eq:Main_Results10}.

\end{theorem}
}

%%%%%%%%%%%%%%%%%%%%%%%%%%%%%%%%%%%%%%%%%%%%%%%%%%%%%%%%%%%%%%%%%%%%%%%%%%%%%%%%%%%%%%%%%%%%%%%%%%%%%%%%%%%%%%%%%%%%%%%%%%
\begin{remark}\label{Remark_1}
 Analyzing relation \eqref{eq:Main_Results12}, we can see that in particular, the equality
  \[
  \lim\limits_{n\to \infty} \mathcal{E}_n ({\mathcal K}_{\lambda,s})_{2,\rho_n}=0
  \]
 is satisfied when $\lambda> 1$ or when $0<\lambda<1$ and  the numbers $a_k$ belong to the set ${\mathbb C}_+\setminus \{z\in {\mathbb C}: |z|<1\}$.

\end{remark}

\begin{remark}\label{Remark_2}

For a given sequence  ${\bf a}=\{a_j\}_{j=1}^{\infty}$, $a_j=\alpha_j+{\mathrm i}\beta_j$, $\beta_j>0$, and any $\lambda>0$, consider the sequence ${\bf z}={\bf z}({\bf a},\lambda)=\{z_j\}_{j=1}^{\infty}$
from the right half-plane such that
 \[
 z_j:=r_j(\cos \varphi_j+{\mathrm i}\sin \varphi_j)=\tsfrac {(\beta_j+\lambda)+{\mathrm i}\alpha_j}{\alpha_j^2+(\beta_j+\lambda)^2}, \quad r_j=|z_j|.
 \]
Using \eqref{eq:Preliminaries3}, for any positive integer $k$,
the products ${\mathrm i}^{k}\nu_{k}({\mathrm i}\lambda,{\bf \bar{a}})$  in the definition of the values
$G_k(\lambda,{\bf a})$  can be expressed as follows:
\begin{eqnarray}\nonumber
   {\mathrm i}^k \nu_k({\mathrm i}\lambda,{\bf \bar{a}})&=& {\mathrm i}^k \sum_{|{\bf k}|=k}
  \prod_{j=1}^n  \tsfrac {1}{(-\alpha_j+{\mathrm i}(\beta_j+\lambda))^{k_j}}
    \\ \nonumber
    &=&  \sum_{|{\bf k}|=k}
  \prod_{j=1}^n  \Big( \tsfrac {\beta_j+\lambda-{\mathrm i}\alpha_j}{\alpha_j^2+(\beta_j+\lambda)^2}\Big)^{k_j}
   \\\nonumber
   &=&
     \sum_{|{\bf k}|=k}
  \prod_{j=1}^n  \overline{z}_j^{k_j}
  \\\nonumber
   &=&
    \sum_{|{\bf k}|=k}
  {\bf r}^{\bf k} \Big(\cos ({\bf k}\cdot \mbox{\boldmath${\bf \varphi}$})-{\mathrm i}\sin ({\bf k}\cdot \mbox{\boldmath${\bf \varphi}$})\Big),
  \end{eqnarray}
   where ${\bf k}\cdot \mbox{\boldmath${\bf \varphi}$}=k_1\varphi_1+\ldots+k_n\varphi_n\ \  {\rm and}\ \ {\bf r}^{\bf k}=r_1^{k_1}\cdot \ldots\cdot r_n^{k_n}$.

   Similarly, for the product
$(-{\mathrm i})^k\overline{\nu}_k({\mathrm i}\lambda,\bar{\bf  {a}})$ we have
   \[
  (-{\mathrm i})^k\overline{\nu}_k({\mathrm i}\lambda,\bar{\bf  {a}})=  \sum_{|{\bf k}|=k}
  \prod_{j=1}^n   z_j^{k_j}= \sum_{|{\bf k}|=k}
  {\bf r}^{\bf k} \Big(\cos ({\bf k}\cdot \mbox{\boldmath${\bf \varphi}$})+{\mathrm i}\sin ({\bf k}\cdot \mbox{\boldmath${\bf \varphi}$})\Big).
 \]

\end{remark}

{

\begin{remark}\label{Remark_3}

Note that the sum on right-hand side of \eqref{eq:Main_Results12} can be represented in the form
 \[
 \sum_{k=0}^{s}\sum_{l=0}^{s}   \textstyle{{l+k \choose  l   }} G_k\overline{G}_l=
 {\displaystyle\int_0^\infty} {\rm e}^{-t}
 {\displaystyle\sum_{k=0}^{s}\sum_{l=0}^{s}}\tsfrac{G_k t^k}{k!} \tsfrac{\overline{G}_l t^l}{l!}{\rm d}t
 =
 {\displaystyle\int_0^\infty} {\rm e}^{-t} \Big|
 {\displaystyle\sum_{k=0}^{s}}\tsfrac{G_k t^k}{k!} \Big|^2{\rm d}t\ge 0 .
 \]

\end{remark}

}

 %%%%%%%%%%%%%%%%%%%%%%%%%%%%%%%%%%%%%%%%%%%%%%%%%%%%%%%%%%%%%%%%%%%%%%%%%%%%%%%%%%%%%%%%%%%%%%%%%

\begin{proof}[Proof of Theorem~\ref{Theo:Main_Results2}] { By virtue of Remark \ref{Rem:Main_Results1}, each polynomial $p\in P_{n-1}$ can be represented as
\[
    p(t)=\tau_n(t,{\bf \bar{a}}) \sum_{k=1}^{n} b_k \Phi_k(t),
\]
with certain coefficients $\{b_k\}_{k=1}^n$.  Thus, taking into account the definition of the function $R_{n,s}^{\lambda}$, we can assert that  \begin{equation}\label{eq:Main_Results14}
  \int_{-\infty}^{\infty}
     \left| \tsfrac{A+Bt}{(t^2+\lambda^2)^{s+1}}-
    p(t) \right|^2 \tsfrac{\mathrm{d}t}{|\rho_n(t)|^2}=\tsfrac{1}{|\rho_0|^{2}} \int_{-\infty}^{\infty}
    \Big|R_{n,s}^{\lambda}(t)- \sum_{k=1}^{n} b_k \Phi_k(t)\Big|^2 \mathrm{d}t.
\end{equation}
Moreover, due to the well-known minimum property of the partial sums of Fourier series, for any quasi-polynomial  $\sum\limits_{k=1}^{n} b_k \Phi_k(t)$ with respect to the system  $\{\Phi_k\}_{k=1}^\infty$, we have
 \begin{equation}\label{eq:Main_Results15}
    \int_{-\infty}^{\infty}
    \Big|R_{n,s}^{\lambda}(t)- \sum_{k=1}^{n} b_k \Phi_k(t)\Big|^2 \mathrm{d}t\ge
     \int_{-\infty}^{\infty} |R_{n,s}^{\lambda}(t)-S_n(R_{n,s}^\lambda)(t)|^2~\mathrm{d}t=:I
\end{equation}
and the infimum on the right-hand side of \eqref{eq:Main_Results11} is realized by the polynomial $p_{n-1}^{(\lambda,s)}$ of the form \eqref{eq:Main_Results10} for which we have
 \[
  \int_{-\infty}^{\infty}
     \left| \tsfrac{A+Bt}{(t^2+\lambda^2)^{s+1}}-
     p_{n-1}^{(\lambda,s)} (t)  \right|^2 \tsfrac{\mathrm{d}t}{|\rho_n(t)|^2}=\tsfrac {1}{|\rho_0|^2 } \int_{-\infty}^{\infty}
    \Big|R_{n,s}^{\lambda}(t)- S_n(R_{n,s}^\lambda)(t)\Big|^2 \mathrm{d}t=\tsfrac {I}{|\rho_0|^2 }.
\]
Therefore, to find the value $ \mathcal{E}_n ({\mathcal K}_{\lambda,s})_{2,\rho_n}$ it is sufficient
 to calculate the integral $I$ on the right-hand side of the relation  \eqref{eq:Main_Results15}.

 Denote
 \begin{equation}\label{eq:Main_Results17}
 W(t):= \sum_{l=0}^{s}   \tsfrac {D_l}{({\mathrm i}\lambda-t)^{l+1}}\quad {\rm and}\quad
 Q(t):=
 \tsfrac{\tau_n(t,{\bf a})}{\tau_n(t,{\bf \bar{a}})}  \overline{W}(t).
\end{equation}
 Then based on \eqref{eq:Main_Results1} we get
\begin{eqnarray}\nonumber
   I=    \int_{-\infty}^{\infty} |W(t)+Q(t) |^2  \mathrm{d}t=\int_{-\infty}^{\infty} \Big|W(t)\Big|^2 \mathrm{d}t+\int_{-\infty}^{\infty} \Big|Q(t)\Big|^2 \mathrm{d}t+2\mathrm{Re}\int_{-\infty}^{\infty}W(t)\overline{Q(t)}\mathrm{d}t.\nonumber
  \end{eqnarray}

For any  $t\in \mathbb{R}$, we have $\Big| \tsfrac{\tau_n(t,{\bf a})}{\tau_n(t,{\bf \bar{a}})}\Big|=1$.
Therefore, taking into account the notation (\ref{eq:Main_Results17}), we conclude that
 \[
  \int_{-\infty}^{\infty} |W(t)|^2\mathrm{d}t=
    \int_{-\infty}^{\infty} |Q(t)|^2\mathrm{d}t.
 \]
The functions
 \[
  \tsfrac {1}
  {({\mathrm i}\lambda-t)^{l}}\quad {\rm and }\quad \tsfrac {1}
  {(-{\mathrm i}\lambda-t)^{l}},\quad l=1,2,\ldots
 \]
 as  functions of the complex variable $t$, are regular in the domains $\overline{\mathbb{C}}_-$  and $\overline{\mathbb{C}}_+$, respectively, and for $|t| \to +\infty$ it is    ${\mathcal O}(|t|^{-l})$. Therefore, based on the Cauchy  theorem
 \[
    \int_{-\infty}^{\infty} Q (t) \overline{W  (t)}~\mathrm{d}t=0.
 \]

 %%%%%%%%%%%%%%%%%%%%%%%%%%%%%%%%%%%%%%%%%%%%%%%%%%%%%%%%%%%%%%%%%%%%%%%%%%%%%%%%%%%%%%%%%%%%%%
 }
Therefore
 \begin{equation}\label{eq:Main_Results18}
 I=2\int_{-\infty}^{\infty} |W(t)|^2\mathrm{d}t.
 \end{equation}

Further, note that for any nonnegative integers $k$ and $j$,
\begin{equation} \label{eq:Main_Results19}
  \int_{-\infty}^{\infty}
  \tsfrac {{\mathrm d}t}  {({\mathrm i}\lambda-t)^{k+1}(-{\mathrm i}\lambda-t)^{j+1}}
  =\tsfrac{\pi}{\lambda}\cdot \tsfrac{ (-1)^{j} {j+k \choose  j   }}{  (2 {\mathrm i} \lambda)^{j+k}} .
\end{equation}
Indeed, the function
 \[
  \tsfrac {1}{({\mathrm i}\lambda-t)^{k+1}(-{\mathrm i}\lambda-t)^{j+1}}=
  \tsfrac {(-1)^{k+j}}
  {(t-{\mathrm i}\lambda)^{k+1}(t+{\mathrm i}\lambda)^{j+1}},\quad k,j=0,1,\ldots
 \]
as a function  of the complex variable $t$ has a pole of the order $k+1$ at the point $t={\mathrm i}\lambda$
in the domain $\mathbb{C}_+$, and for $|t|\to +\infty$, it has the order   ${\mathcal O}(|t|^{-k-j-2})$. Therefore, based on the Cauchy residue theorem, we get \eqref{eq:Main_Results19}.

Assume that $a_j=\alpha_j+{\mathrm i}\beta_j$, $\beta_j>0$, $j=1,2,\ldots$. Then, taking into account the accepted notations
we see that
  \begin{equation} \label{eq:Main_Results20}
    |\tau_n({\mathrm{i}\lambda},{\bf \bar{a}})|^2=| {\tau}_n({-\mathrm{i}\lambda},{\bf a})|^2=\frac{|\rho_n(-\mathrm i\lambda)|}{|\rho_0|^2}.
  \end{equation}
 { Combining relations   \eqref{eq:Main_Results18},  \eqref{eq:Main_Results1_0}, \eqref{eq:Main_Results17}, \eqref{eq:Main_Results20} and \eqref{eq:Main_Results19}, we obtain the necessary equality for $I$:
 \[
 I=\tsfrac{2 \pi}{ \lambda}\sum_{k=0}^{s}\sum_{l=0}^{s} \tsfrac{ (-1)^{l} {l+k \choose  l   }}{  (2 {\mathrm i} \lambda)^{l+k}} D_k\overline{D}_l
 =\tsfrac{4 \pi}{(2 \lambda)^{2s+3}\mu_n(\lambda, {\bf a})}\sum_{k=0}^{s}\sum_{l=0}^{s}   \textstyle{l+k \choose  l} G_k\overline{G}_l,
 \]
where $G_k$ is given by \eqref{eq:Main_Results13}.}

\end{proof}

\vskip 3mm

{\bf\textsl Acknowledgments.} This work was partially supported by a grant from the Simons Foundation (SFI-PD-Ukraine-00014586, A.Sh.).

%%%%%%%%%%%%%%%%%%%%%%%%%%%%%%%%%%%%%%%%%%%%%%%%%%%%%%%%%%%%%%%%%%%%%%%%%%%%%%%%%%%%%%%%%%%%%%%%%%%%%%%%%%%%%%%%%%%%%%%%

\end{document}